\documentclass[]{article}

\usepackage{graphicx}
\usepackage{amsmath}
\usepackage{amssymb}
\usepackage{amsthm}
\usepackage{pxfonts}
\usepackage{enumerate}
\usepackage{color}
\usepackage{mathdots}
\usepackage{sectsty}
\usepackage[hidelinks]{hyperref}
\usepackage{tikz}
\usepackage{caption}
\usepackage{adjustbox}

\sectionfont{\scshape\centering\fontsize{11}{14}\selectfont}
\subsectionfont{\scshape\fontsize{11}{14}\selectfont}
\usepackage{fancyhdr}

\newcommand\shorttitle{On the rank and crank statistics for integer partitions}
\newcommand\authors{Nian Hong Zhou}

\hypersetup{
    colorlinks=true,       
    linkcolor=blue,          
    citecolor=cyan,        
}
\usepackage{verbatim}
\usepackage{url}

\makeatother

\newtheorem{theorem}{Theorem}[section]
\newtheorem{lemma}{Lemma}[section]
\newtheorem{corollary}[theorem]{Corollary}
\newtheorem{proposition}[lemma]{Proposition}

\theoremstyle{remark}
\newtheorem{remark}[lemma]{Remark}
\usepackage{tikz}
\usetikzlibrary{decorations.markings}
\makeatother

\def\nb{\mathbb N}
\def\zb{\mathbb Z}

\def\cb{{\mathbb C}}

\def\rrw{\rightarrow}
\numberwithin{equation}{section}

\title{\large \bf ON THE DISTRIBUTION OF RANK AND CRANK STATISTICS FOR INTEGER PARTITIONS}
\author{\small NIAN HONG ZHOU}

\date{} 

\begin{document}
\maketitle
\begin{abstract}
Let $k$ be a positive integer and $m$ be an integer. Garvan's $k$-rank $N_k(m,n)$ is the number of partitions of $n$ into at least $(k-1)$ successive Durfee squares with $k$-rank equal to $m$. In this paper give some asymptotics for $N_k(m,n)$ with $|m|\ge \sqrt{n}$ as $n\rightarrow \infty$. As a corollary,
we give a more complete answer for the Dyson's crank distribution conjecture. We also establish some asymptotic formulas for finite differences of $N_k(m,n)$ with respect to $m$ with $m\gg \sqrt{n}\log n$.
\end{abstract}

\maketitle

\section{Introduction and statement of results}
A partition of an integer $n$ is a sequence of non-increasing positive integers whose sum equals $n$. As usual, let $p(n)$ be the number of integer partitions of integer $n$. It is well known that the generating function for $p(n)$ is given by
$$\sum_{n\in\zb}p(n)q^n=\frac{1}{(q;q)_{\infty}},$$
where $(a;q)_{\infty}=\prod_{j\ge 0}(1-aq^j)$ for any $a\in\cb$ and $|q|<1$.

In 1944, F.J. Dyson  \cite{MR3077150} introduced the rank statistic for integer partitions. The rank of a partition was introduced to explain Ramanujan's famous partition congruences with modulus $5$ and $7$. If $N(m, n)$ denotes the number of partitions of $n$ with rank $m$, then it is well known that
$$
\sum_{n\ge 0}N(m,n)q^n=\frac{1}{(q;q)_{\infty}}\sum_{n\ge 1}(-1)^{n-1}q^{n(3n-1)/2+|m|n}(1-q^n).
$$
However the rank fails to explain Ramanujan¡¯s congruence modulo $11$.  Therefore Dyson conjectured the existence of another statistic which he called the ¡°crank¡± which would explain Ramanujan¡¯s partition congruence modulo $11$. The crank was found by G.E. Andrews  and F.G. Garvan \cite{MR929094, MR920146}. Let $M(m,n)$ be the number of partitions of $n$ with crank $m$, then
$$
\sum_{n\ge 0}M(m,n)q^n=\frac{1}{(q;q)_{\infty}}\sum_{n\ge 1}(-1)^{n-1}q^{n(n-1)/2+|m|n}(1-q^n).
$$

In \cite{MR1001259}, Dyson gave the following asymptotic formulae conjecture for the crank statistic for integer partitions:
\begin{equation}\label{dysc}
M\left( m,n \right)\sim \frac{\pi}{4\sqrt{6n}} {\rm sech}^2 \left( \frac{\pi m}{2\sqrt{6n}}   \right) p(n), n\rrw +\infty.
\end{equation}
He then asked the question about the precise range of $m$ in \eqref{dysc}, which holds and about the error term.

The case of fixed $m$ of \eqref{dysc} has been investigated by R. Mao \cite{MR3103192}, K. Bringmann and J. Manschot \cite{MR3210725}, and B. Kim, E. Kim and J. Seo \cite{MR3279269}. In \cite{MR3451872}, K. Bringmann and J. Dousse proved
\begin{theorem}\label{mnth}\footnote{In \cite{MR3451872}, the $O$-term of \eqref{dysc} is $O\left(|m|^{1/3}n^{-1/4}\right)$, the case of $m=0$ was missed.}\label{dysm}For $|m|\le \frac{1}{\pi}\sqrt{\frac{n}{6}}\log n$, we have as $n$ tends to infinity,
\begin{equation}\label{dysc1}
M\left( m,n \right)=\frac{\pi}{4\sqrt{6n}} {\rm sech}^2 \left( \frac{\pi m}{2\sqrt{6n}}   \right) p(n)\left(1+O\left(\frac{|m|^{1/3}+1}{n^{1/4}}\right)\right).
\end{equation}
\end{theorem}

Let $N_k(m, n)$ be the number of partitions of $n$ into at least $(k-1)$ successive Durfee squares with $k$-rank equal to $m\in\zb$ and $k\in\zb_+$. F.G. Garvan \cite{MR1291125} proved
\begin{equation}\label{fk}
\sum_{n\ge 0}N_k(m,n)q^n=\frac{1}{(q;q)_{\infty}}\sum_{n\ge 1}(-1)^{n-1}q^{n((2k-1)n-1)/2+|m|n}(1-q^n).
\end{equation}
It is clear that $M(m,n)=N_1(m,n)$ and $N(m,n)=N_2(m,n)$.

In  \cite{MR3337213}, J. Dousse and M.H. Mertens proved that the same asymptotic formula \eqref{dysc1} holds for $N(m,n)$. D. Parry and  R.C. Rhoades \cite{MR3565363} proved that
\begin{theorem}\footnote{It is need to note that the statement of the main result \cite[Theorem 1.2]{MR3565363} has some miss. Here is the corrected version.}
For $\sqrt{n}\ll |m|= o(n^{3/4})$ with $n\rrw \infty$, one has
\begin{equation}\label{dr}
N_k(m,n)\sim \frac{\pi}{\sqrt{6n}}\left(e^{\frac{\pi (m+k)}{2\sqrt{6n}}}+e^{-\frac{\pi (m+k)}{2\sqrt{6n}}}\right)^{-2}p(n).
\end{equation}
\end{theorem}

In this paper, we shall let $k$ be a fixed positive integer and focus on the asymptotics of Garvan $k$-rank $N_k(m,n)$ with $|m|\ge \sqrt{n}$ and $n$ tends to infinity. Our main result as follows.

\begin{theorem}\label{main}Let $F_k(m,n)=p(n-|m|-k+1)-p(n-|m|-k)$. We have:
\begin{enumerate}
  \item If $m\ge (n+3)/2-2k$ then
$$N_k(m,n)=F_k(m,n).$$
  \item If $\sqrt{n}\le |m|<n/2$ then
  $$N_k(m,n)=F_k(m,n)\left(1+O\left(e^{-\frac{\pi |m|}{\sqrt{6n}}}+\sqrt{n}e^{-n^{1/3}}\right)\right).$$
\end{enumerate}
In particular,
\[N_k(m,n)=\frac{\pi}{\sqrt{6}}\frac{p(n-|m|)}{\sqrt{n-|m|}}\left(1+O\left(\frac{1}{\sqrt{n-|m|}}+e^{-\frac{\pi |m|}{\sqrt{6n}}}\right)\right)\]
holds for all $m$ such that $|m|\ge \sqrt{n}$ and $n-|m|\rrw +\infty$.
\end{theorem}
From the asymptotics for $M(m,n)$ and $N(m,n)$, that is Theorem \ref{mnth} for $M(m,n)$ and $N(m,n)$,
and above Theorem \ref{main}, we have the following corollary.
\begin{corollary}\label{cor1} Let $k=1, 2$ and $m,n\in\zb$. We have as $n-|m|\rrw+\infty$,
\begin{equation}
\frac{N_k(m,n)}{\widetilde{M}(m,n)}-1\ll \frac{1}{\sqrt{n-|m|}}+\min\left(e^{-\frac{\pi |m|}{\sqrt{6n}}}, \frac{|m|^{1/3}+1}{n^{1/4}}\right),
\end{equation}
where
$$\widetilde{M}(m,n)=\frac{\pi}{\sqrt{6}}\left(1+e^{-\frac{\pi|m|}{\sqrt{6n}}}\right)^{-2}\frac{p(n-|m|)}{\sqrt{n-|m|}}.$$
In particular, as $n-|m|\rrw+\infty$,
$$N_k(m,n)\sim \widetilde{M}(m,n).$$
\end{corollary}
From Theorem \ref{main} and Corollary \ref{cor1}, we shall give a more complete answer for the Dyson's crank distribution conjecture.
\begin{theorem}\label{zn1} Dyson's asymptotic formulae \eqref{dysc} is valid if and only if $|m|=o(n^{3/4})$ as $n$ tends to infinity.  More precisely,
\begin{equation*}
\frac{M\left(m,n\right)}{p(n)}=\frac{\pi}{4\sqrt{6n}} {\rm sech}^2 \left( \frac{\pi m}{2\sqrt{6n}}   \right) \left(1+O\left(\min\left(e^{-\frac{\pi |m|}{\sqrt{6n}}}, \frac{|m|^{1/3}+1}{n^{1/4}}\right)+\frac{m^2}{n^{3/2}}\right)\right).
\end{equation*}
\end{theorem}
\begin{remark}
The error term of Theorem \ref{zn1} smaller than that
in Theorem \ref{dysc} when $|m|> \sqrt{n}\left(\log n-4\log\log n\right)/(\pi\sqrt{6})$ as $n\rrw \infty$.
\end{remark}
As another application of Theorem \ref{main}, by a straightforward calculation, we obtain the following asymptotic result for finite differences of $N_k(n,m)$ with respect to $m$ with $m\gg \sqrt{n}\log n$.
\begin{corollary}\label{eqic}Let $m,n, r, k\in\nb$ and $r, k$ be fixed. If $m\ge c\sqrt{n}\log n$ with $c>\frac{(r+1)\sqrt{6}}{2\pi}$ be fixed and $n-m\rrw \infty$, then
\begin{equation}\label{eqi}
\sum_{j=0}^{r}(-1)^j\binom{r}{j}N_{k}(m+j,n)\sim \left(\frac{\pi}{\sqrt{6(n-m)}}\right)^{r+1}p(n-m).
\end{equation}
In particular,
\begin{equation*}
N_{k}(m,n)-N_k(m+1,n)\sim \frac{\pi^2}{\sqrt{6(n-m)}}p(n-m)
\end{equation*}
holds for $m\ge c_1\sqrt{n}\log n$ with $c_1>{\sqrt{6}}/{\pi}$ be fixed and $n-m\rrw \infty$.
\end{corollary}
\begin{remark}We can improve \eqref{eqi} and give the size of the error term by Theorem \ref{main}.
We note that the monotonicity properties of $N(m,n)$ have been investigated by S.H. Chan and R. Mao \cite{MR3190432}.
\end{remark}
The proof of Theorem \ref{main} is based on the generating function \eqref{fk} of $N_k(m,n)$ and the Hardy-Ramanujan asymptotic formula for $p(n)$. We present the relation between $N_k(m,n)$ and $p(n)$, and some basic properties of $p(n)$ in Section \ref{s2}. In Section \ref{s3}, we prove the main results of this paper.

\paragraph{Acknowledgements.}The author would like to thank the anonymous referees for their very helpful
comments and suggestions. The author also thank Professor Zhi-Guo Liu for his consistent encouragement.

\section{Preliminaries}\label{s2}

\subsection{Basic properties for Garvan $k$-ranks}~

A key ingredient of our asymptotic results is find the connection $N_k(m,n)$ with $p(n)$. Note that $N_k(m,n)=N_k(-m,n)$, we just need consider the case of $m\ge 0$. We prove the following representation for $N_k(m,n)$.
\begin{proposition}\label{pro1}Let $N_k(m,n)$ be defined as \eqref{fk}, integers $n>0$ and $m\ge 0$. We have
$$
N_k(m,n)=\sum_{\ell\ge 1}(-1)^{\ell-1}F_k(\ell;m,n),
$$
where
$$F_k(\ell;m,n)=p\left(n-m\ell-Q_k(\ell)\right)-p\left(n-m\ell-Q_k(\ell)-\ell\right)$$
with $Q_k(\ell)=(k-1/2)\ell^2-\ell/2$.
\end{proposition}
\begin{proof}
For $m\ge 0$, from
\begin{align*}
\sum_{n\ge 0}N_k(m,n)q^n=\frac{1}{(q;q)_{\infty}}\sum_{\ell\ge 1}(-1)^{\ell-1}q^{\ell((2k-1)\ell-1)/2+m\ell}(1-q^{\ell})
\end{align*}
we have
\begin{align*}
\sum_{n\ge 0}N_k(m,n)q^n&=\sum_{\ell\ge 1}\sum_{r\in\zb}(-1)^{\ell-1}p(r)q^{r+\ell(\ell(2k-1)-1)/2+m\ell}(1-q^{\ell})\\
&=\sum_{n\ge 1}q^n\left(\sum_{\substack{\ell\ge 1, r\in\zb\\ \frac{\ell((2k-1)\ell-1)}{2}+m\ell+r=n}}-\sum_{\substack{\ell\ge 1, \ell\in\zb\\ \frac{\ell((2k-1)\ell+1)}{2}+m\ell+r=n}}\right)(-1)^{\ell-1}p(r).
\end{align*}
Then, comparing the coefficients of $q$ on both sides we complete the proof.
\end{proof}
We need the following Hardy--Ramanujan asymptotic result for $p(n)$, which can be found in \cite{MR1575586}.
\begin{lemma}\label{lem1}We have for $n\in\zb_+$,
\begin{equation*}
p(n)=\frac{e^{B\sqrt{n-1/24}}}{4\sqrt{3}(n-1/24)}\left(1-\frac{1}{B\sqrt{n-1/24}}\right)+ O\left(n^{-1}e^{B\sqrt{n}/2}\right),
\end{equation*}
where and throughout this section $B=2\pi/\sqrt{6}$.
\end{lemma}

We also need the following approximation for $p(n+r)$ with $r=O(\sqrt{n})$.

\begin{lemma}\label{lem2}
Let integer $r= O\left(\sqrt{n}\right)$ and $n$ be a positive integer. We have as $n\rrw \infty$,
$$
\frac{p(n+r)-p(n)}{p(n)}=\frac{Br}{\sqrt{n}}+O\left(\frac{1+|r|+|r|^2}{n}\right).
$$
In particular, as $n\rrw \infty$,
$$
\frac{p(n+r)-p(n)}{p(n)}\ll \frac{|r|+1}{\sqrt{n}}.
$$
\end{lemma}
\begin{proof} It is clear that
\begin{align*}
\frac{p(n+r)}{p(n)}&=e^{B\left(\sqrt{n+r-1/24}-\sqrt{n-1/24}\right)}\left(1+O\left(\frac{|r|}{n}\right)\right)\\
&=e^{B\sqrt{n}\left(\frac{r-1/24}{2 n}-\frac{-1/24}{2 n}+O\left(\frac{|r|^2+1}{n^2}\right)\right)}\left(1+O\left(\frac{|r|}{n}\right)\right)\\
&=\left(1+\frac{Br}{2\sqrt{n}}+O\left(\frac{|r|^2+|r|+1}{n}\right)\right)\left(1+O\left(\frac{|r|}{n}\right)\right).
\end{align*}
by generalized binomial theorem, which completes the proof of the lemma.
\end{proof}

\section{The proof of the results of this paper}\label{s3}

We prove Theorem \ref{main} in Subsection \ref{s31}, Corollary \ref{cor1} and Theorem \ref{zn1} in Subsection \ref{s32}, and Corollary \ref{eqic} in the last subsection of this section.
\subsection{The proof of Theorem \ref{main}}\label{s31}~

Let $m\ge 0$ and $n\ge k-1$. From Proposition \ref{pro1}, it is easy to prove that
\begin{equation}\label{cas1}
N_k(m,n)=F_k(1;m,n)=F_k(m,n)
\end{equation}
for $m\ge (n+3)/2-2k$.  For $n/(4k)\le m\le n/2$,
\begin{align*}
E_k(m,n)&:=N_k(m,n)-F_k(m,n)\\
&\ll \sum_{2\le \ell\le n/m}p(n-m\ell)\ll p(n-2m)\ll p(n-m-\lfloor n/(4k)\rfloor)\\
&\ll n^{-1}e^{B\sqrt{n-m-n/(4k)+O(1)}}\ll e^{-n^{1/3}}p(n-m),
\end{align*}
by use Lemma \ref{lem1}, the asymptotics for $p(n)$. For $n^{1/2}\le m\le n/(4k)$,
we note that
$$0\le m\ell+Q_k(\ell)\le \frac{n}{2k}+k\left(\frac{n}{2km}\right)^2+O\left(\frac{n}{m}\right)\le \frac{3}{4k}n+O(\sqrt{n}),$$
if $0\le \ell\le n/(2km)$. Then by Lemma \ref{lem2} we have
\begin{align*}
E_k(m,n)&:=N_k(m,n)-F_k(m,n)\\
=&\sum_{2\le \ell \le n/(2km)}(-1)^{\ell-1}F_k(\ell;m,n)+O\left(
\sum_{\ell>n/(2km)}p(n-\ell m)\right)\\
\ll &\sum_{2\le \ell \le n/(2km)} p(n-m\ell-Q_k(\ell))\frac{\ell+1}{\sqrt{n-m\ell-Q_k(\ell)}}+\frac{n}{m}p(n-\lfloor n/(2k)\rfloor)\\
\ll &\sum_{2\le \ell \le n/(2km)}\frac{\ell}{\sqrt{n}}p(n-m\ell)+\sqrt{n}p(n-\lfloor n/(2k)\rfloor).
\end{align*}
By use the asymptotics for $p(n)$ and $\sqrt{n}\le m\le n/(4 k)$, we further have
\begin{align*}
E_k(m,n)
\ll &\sum_{2\le \ell \le n/(2km)}\frac{\ell}{n^{3/2}}e^{B\sqrt{n-m\ell}}+n^{-1/2}e^{B\sqrt{n-\lfloor n/(2k)\rfloor}}\\
\ll &\frac{e^{B\sqrt{n-m}}}{n^{3/2}}\sum_{2\le \ell \le n/(2m)}\ell e^{B(\sqrt{n-m\ell}-\sqrt{n-m})}+\frac{e^{B\sqrt{n-m+m-n/(2k)+O(1)}}}{\sqrt{n}}\\
\ll &\frac{p(n-m)}{\sqrt{n}}\sum_{2\le \ell\le n/(2m)}\ell e^{-\frac{B(\ell-1)m}{2\sqrt{n}}}+n^{-1/2}e^{B\sqrt{n-m-n/(4k)+O(1)}}\\
\ll &\frac{p(n-m)}{\sqrt{n}}+\frac{e^{B\sqrt{n/2}}}{\sqrt{n}}\ll \left(\frac{e^{-\frac{Bm}{2\sqrt{n}}}}{\sqrt{n}}+e^{-n^{1/3}}\right)p(n-m).
\end{align*}
Further, from Lemma \ref{lem2} we have as $n-|m|\rrw +\infty$,
\begin{align}\label{eq9}
F_k(m,n)&=p(n-(|m|+k)+1)-p(n-(|m|+k))\nonumber\\
&=\frac{B}{2}\frac{p(n-|m|)}{\sqrt{n-|m|}}\left(1+O\left(\frac{1}{\sqrt{n-|m|}}\right)\right).
\end{align}
Therefore, for $\sqrt{n}\ll m<n/2$, from above we find that
\begin{equation}\label{cas4}
\frac{N_k(m,n)-F_k(m,n)}{F_k(m,n)}\ll e^{-\frac{B|m|}{2\sqrt{n}}}+\sqrt{n}e^{-n^{1/3}},
\end{equation}
which completes the proof of Theorem \ref{main} by substitute $B=2\pi/\sqrt{6}$.
\subsection{The proof of Corollary \ref{cor1} and Theorem \ref{zn1}}\label{s32}~
\subsubsection{The proof of Corollary \ref{cor1}}
For $m=o(n^{3/4})$, by Lemma \ref{lem1} it is easy to find that as $n\rrw \infty$,
$$\frac{p(n-|m|)}{\sqrt{n-|m|}}=\frac{p(n)}{\sqrt{n}}e^{-\frac{B|m|}{2\sqrt{n}}}\left(1+O\left(\frac{n+|m|^2}{n^{3/2}}\right)\right)$$
Hence we have for $m=o(n^{3/4})$,
\begin{align}\label{mmm}
\widetilde{M}(m,n)&=\frac{\pi}{\sqrt{6}}\left(1+e^{-\frac{\pi|m|}{\sqrt{6n}}}\right)^{-2}\frac{p(n-|m|)}{\sqrt{n-|m|}}\nonumber\\
&=\frac{\pi}{4\sqrt{6n}}{\rm sech}^2\left(\frac{\pi |m|}{2\sqrt{6n}}\right)p(n)\left(1+O\left(\frac{n+|m|^2}{n^{3/2}}\right)\right).
\end{align}
From Theorem \ref{mnth} for $M(m,n)$ and $N(m,n)$ and above we have for $|m|\le (\sqrt{n}\log n)/(\pi\sqrt{6})$,
\begin{align*}
\widetilde{M}(m,n)=N_k(m,n)\left(1+O\left(\frac{|m|^{1/3}+1}{n^{1/4}}\right)\right).
\end{align*}
For $|m|\gg \sqrt{n}\log n$, we have
\begin{align*}
\widetilde{M}(m,n)&=\left(1+O\left(e^{-\frac{\pi|m|}{\sqrt{6n}}}\right)\right)\frac{\pi}{\sqrt{6}}\frac{p(n-|m|)}{\sqrt{n-|m|}}\\
&=N_k(m,n)\left(1+O\left(\frac{1}{\sqrt{n-|m|}}+e^{-\frac{\pi |m|}{\sqrt{6n}}}\right)\right)
\end{align*}
by Theorem \ref{main}.  From above we complete the proof of Corollary \ref{cor1}.
\subsubsection{The proof of Theorem \ref{zn1}}
It is clear that Theorem \ref{zn1} follows from Theorem \ref{main}, Corollary \ref{cor1}, \eqref{mmm} and the following lemma by substitute $B=2\pi/\sqrt{6}$.
\begin{lemma}\label{dprz} Let $0\le m\le n$. Then as $n\rrw \infty$,
\begin{equation}\label{zzz}
\frac{p(n-m+1)-p(n-m)}{p(n)}\sim \frac{B}{8\sqrt{n}}\mathrm{sech}^2\left(\frac{B m}{4\sqrt{n}}\right)
\end{equation}
if and only if $1/m=o(n^{-1/2})$ and $m=o(n^{3/4})$.
\end{lemma}
\begin{proof} For $n\ge m\ge 8 n/9$, we have
\[\frac{p(n-m+1)-p(n-m)}{p(n)}\ll \frac{e^{B \sqrt{n/9}}}{n^{-1}e^{B\sqrt{n}}}= ne^{-\frac{2B}{3}\sqrt{n}}\]
by Lemma \ref{lem1} and
\[\frac{B}{8\sqrt{n}}\mathrm{sech}^2\left(\frac{B m}{4\sqrt{n}}\right)\gg n^{-1}e^{-\frac{Bm}{2\sqrt{n}}}\gg n^{-1}e^{-\frac{B}{2}\sqrt{n}}\]
by the definition of $\mathrm{sech}(x)$. Therefore, if \eqref{zzz} is valid then $m< 8n/9$. If $0\le m\le 8n/9$ then
\begin{equation}\label{eq10}
\frac{p(n-m+1)-p(n-m)}{p(n)}=\left(1+O\left(\frac{1}{\sqrt{n}}\right)\right)\frac{Be^{-\frac{Bm}{\sqrt{n-m}+\sqrt{n}}}}{2\sqrt{n}}\left(1-\frac{m}{n}\right)^{-3/2}
\end{equation}
by Lemma \ref{lem1} and \eqref{eq9}, and note that
\begin{equation}\label{eq11}
\frac{B}{8\sqrt{n}}\mathrm{sech}^2\left(\frac{B m}{4\sqrt{n}}\right)=\frac{B}{2\sqrt{n}}e^{-\frac{Bm}{2\sqrt{n}}}\frac{1}{\left(1+e^{-\frac{B m}{8\sqrt{n}}}\right)^2}.
\end{equation}
Hence clearly, if \eqref{zzz} is valid then $m=o(n)$.  Let $0\le m=o(n)$, from \eqref{eq10} we have
\begin{equation}\label{eq12}
\frac{p(n-m+1)-p(n-m)}{p(n)}=\left(1+O\left(\frac{1}{\sqrt{n}}\right)\right)\frac{B}{2\sqrt{n}}e^{-\frac{Bm}{2\sqrt{n}}-\frac{m^2}{8n^{3/2}}(1+o(1))}.
\end{equation}
by generalized binomial theorem. Combining \eqref{eq11} and \eqref{eq12} we immediately obtain that \eqref{zzz} is valid if and only if $1/m=o(n^{-1/2})$ and $m=o(n^{3/4})$.
\end{proof}
\subsection{The proof of the Corollary \ref{eqic}}~

First of all, let $m\ge \sqrt{n}$ and $n-m$ tends to infinity. We have
\begin{align}\label{eq14}
N_k(m,n)&=F_k(m,n)\left(1+O\left(e^{-\frac{\pi m}{\sqrt{6n}}}+e^{-n^{1/3}}\right)\right)\nonumber\\
&=p(n-(m+k)+1)-p(n-(m+k))+ R_k(m,n)
\end{align}
with
\[R_k(m,n)\ll \left(e^{-\frac{\pi m}{\sqrt{6n}}}+\sqrt{n}e^{-n^{1/3}}\right)p(n-m)\]
by Theorem \ref{main}. Now, from \cite[Equ. (1.2)]{MR932523} we have for $r\in\zb_{\ge 0}$ be fixed and as $N\rrw \infty$
\begin{equation}\label{eq13}
\Delta_N^{r} p(N)=\left(\frac{\pi}{\sqrt{6N}}\right)^{r}p(N)\left(1+O\left(\frac{1}{\sqrt{N}}\right)\right),
\end{equation}
where for a function $f(x)$,
\[\Delta_x^0f(x):=f(x),\;\; \Delta_x f(x):=f(x)-f(x-1)\]
and
\[\Delta_x^{\ell}p(x):=\Delta_x(\Delta_x^{\ell-1}p(x))~\mbox{for}~\ell\in\nb \]
is the backward difference. Hence clearly,
\begin{equation}\label{eq17}
(-1)^r\Delta_m^{r} N_k(m+r,n)=\sum_{j=0}^{r}(-1)^j\binom{r}{j}N_{k}(m+j,n).
\end{equation}
By \eqref{eq14} we obtain that
\begin{align}\label{eq18}
\Delta_m^{r} N_k(m+r,n)=&-\Delta_m^{r+1} p(n-((m+r)+k))\nonumber\\
&+O_{k,r}\left(\left(e^{-\frac{\pi m}{\sqrt{6n}}}+e^{-\frac{\pi \sqrt{n/6}}{5}}\right)\sum_{j=0}^{r+1}p(n-m+j)\right).
\end{align}
Moreover, we have
\begin{equation}\label{eq19}
\Delta_m^{r+1} p(n-((m+r)+k))=(-1)^{r+1}\Delta_j^{r+1}p(j)\bigg|_{j=n-m-k+1}.
\end{equation}
Using \eqref{eq13}--\eqref{eq19} and Lemma \ref{lem1}, it is clear that
\begin{align*}
\left(\frac{\sqrt{6(n-m)}}{\pi}\right)^{r+1}\frac{(-1)^r\Delta_m^{r} N_k(m+r,n)}{p(n-m)}-1&\ll_{k, r} \frac{1}{\sqrt{n-m}}+\frac{(n-m)^{\frac{r+1}{2}}}{e^{\frac{\pi m}{\sqrt{6n}}}}\\
&\ll \frac{1}{\sqrt{n-m}}+ n^{\frac{r+1}{2}}e^{-\frac{\pi m}{\sqrt{6n}}}
\end{align*}
and the proof of the corollary follows.

\bigskip
\noindent
{\sc School of Mathematical Sciences,  East China Normal University\\
500 Dongchuan Road, Shanghai 200241, PR China}\newline
\href{mailto:nianhongzhou@outlook.com}{\small nianhongzhou@outlook.com}


\begin{thebibliography}{10}

\bibitem{MR3077150}
F.~J. Dyson.
\newblock Some guesses in the theory of partitions.
\newblock {\em Eureka}, (8):10--15, 1944.

\bibitem{MR929094}
George~E. Andrews and F.~G. Garvan.
\newblock Dyson's crank of a partition.
\newblock {\em Bull. Amer. Math. Soc. (N.S.)}, 18(2):167--171, 1988.

\bibitem{MR920146}
F.~G. Garvan.
\newblock New combinatorial interpretations of {R}amanujan's partition
  congruences mod {$5,7$} and {$11$}.
\newblock {\em Trans. Amer. Math. Soc.}, 305(1):47--77, 1988.

\bibitem{MR1001259}
Freeman~J. Dyson.
\newblock Mappings and symmetries of partitions.
\newblock {\em J. Combin. Theory Ser. A}, 51(2):169--180, 1989.

\bibitem{MR3103192}
Renrong Mao.
\newblock Asymptotic inequalities for {$k$}-ranks and their cumulation
  functions.
\newblock {\em J. Math. Anal. Appl.}, 409(2):729--741, 2014.

\bibitem{MR3210725}
Kathrin Bringmann and Jan Manschot.
\newblock Asymptotic formulas for coefficients of inverse theta functions.
\newblock {\em Commun. Number Theory Phys.}, 7(3):497--513, 2013.

\bibitem{MR3279269}
Byungchan Kim, Eunmi Kim, and Jeehyeon Seo.
\newblock Asymptotics for {$q$}-expansions involving partial theta functions.
\newblock {\em Discrete Math.}, 338(2):180--189, 2015.

\bibitem{MR3451872}
Kathrin Bringmann and Jehanne Dousse.
\newblock On {D}yson's crank conjecture and the uniform asymptotic behavior of
  certain inverse theta functions.
\newblock {\em Trans. Amer. Math. Soc.}, 368(5):3141--3155, 2016.

\bibitem{MR1291125}
Frank~G. Garvan.
\newblock Generalizations of {D}yson's rank and non-{R}ogers-{R}amanujan
  partitions.
\newblock {\em Manuscripta Math.}, 84(3-4):343--359, 1994.

\bibitem{MR3337213}
Jehanne Dousse and Michael~H. Mertens.
\newblock Asymptotic formulae for partition ranks.
\newblock {\em Acta Arith.}, 168(1):83--100, 2015.

\bibitem{MR3565363}
Daniel Parry and Robert~C. Rhoades.
\newblock On {D}yson's crank distribution conjecture and its generalizations.
\newblock {\em Proc. Amer. Math. Soc.}, 145(1):101--108, 2017.

\bibitem{MR3190432}
Song~Heng Chan and Renrong Mao.
\newblock Inequalities for ranks of partitions and the first moment of ranks
  and cranks of partitions.
\newblock {\em Adv. Math.}, 258:414--437, 2014.

\bibitem{MR1575586}
G.~H. Hardy and S.~Ramanujan.
\newblock Asymptotic {F}ormulaae in {C}ombinatory {A}nalysis.
\newblock {\em Proc. London Math. Soc. (2)}, 17:75--115, 1918.

\bibitem{MR932523}
A.~M. Odlyzko.
\newblock Differences of the partition function.
\newblock {\em Acta Arith.}, 49(3):237--254, 1988.

\end{thebibliography}
\end{document}